\documentclass[10pt,a4paper]{amsart}
\usepackage[english]{babel}
\usepackage{amssymb,xypic,amscd}
\CompileMatrices

\newtheorem{defn}{Definition}
\newtheorem{thm}[defn]{Theorem}

\newtheorem{cor}[defn]{Corollary}
\newtheorem{lem}[defn]{Lemma}
\newtheorem{prop}[defn]{Proposition}

\theoremstyle{remark}
\newtheorem{rem}[defn]{Remark}
\theoremstyle{remark}
\newtheorem{exam}{Example}

\numberwithin{equation}{section} \numberwithin{defn}{section}

\newcommand\aut{\operatorname{Aut}}
\newcommand\ed{\operatorname{End}}

\newcommand\Ker{\operatorname{Ker}}
\renewcommand\dim{\operatorname{dim}}

\newcommand\Det{\operatorname{det}}

\newcommand\limpl[1]{\underset{#1}\varprojlim\,}

\def\mod #1/#2{\kern.06em{\raise1.2pt\hbox{$#1$}}/
      {\raise-1.2pt\hbox{$#2$}}}

\newcommand\B{{\mathcal B}}

\renewcommand\tilde{\widetilde}

\renewcommand\lim{\limpl{A\in\B}}
\newcommand\beq{
      \setcounter{equation}{\value{defn}}\addtocounter{defn}1
      \begin{equation}}

\begin{document}

\title[Classification of Endomorphisms with Annihilator]{Classification of Endomorphisms \\ with an Annihilating Polynomial \\ on Arbitrary Vector Spaces}
\author{Fernando Pablos Romo}
\address{Departamento de
Matem\'aticas and Instituto Universitario de F\'{\i}sica Fundamental y Matem\'aticas, Universidad de Salamanca, Plaza de la Merced 1-4,
37008 Salamanca, Espa\~na} \email{fpablos@usal.es}
\thanks{
\\ This work is partially supported by the
Spanish Government research contract no. MTM2012-32342.}

\maketitle

\begin{abstract} The aim of this work is to offer a solution to the problem of the classification of endomorphisms with an annihilating polynomial on arbitrary vector spaces. For these endomorphisms we provide a family of invariants that allows us to classify them when the group of automorphisms acts by conjugation.
\end{abstract}

\bigskip

\setcounter{tocdepth}1

\tableofcontents
\bigskip

\section{Introduction}

The classification of mathematical objects is a classical problem: try to determine the structure of a quotient set up to some equivalence. The classification of endomorphisms on finite-dimensional vector spaces, where the group of automorphisms acts by conjugation, is well-known.

In this work we generalize this classification to endomorphisms admitting an annihilating polynomial on arbitrary vector spaces. This classification also generalizes the solution of the classification problem for finite potent endomorphisms that has been recently provided by the author in \cite{Pa}. As far as we know, a solution of the classification problem for the set of all endomorphisms with an annihilating polynomial is not stated explicitly in the literature.

The paper is organized as follows. In section \ref{s:pre} we briefly recall the basic definitions of this work: the statement of the classification problem for endomorphisms, and the well-known theory of the classification of endomorphisms on finite-dimensional vector spaces, including the description of a method to construct Jordan bases in this case.

Section \ref{s:main-classif} is devoted to giving the main result of this work. Indeed, we offer invariants to classify to classify endomorphisms with an annihilating polynomial on arbitrary vector spaces (Theorem \ref{th:class-anh}). As an example we offer the explicit description of the quotient set obtained from the classification of these endomorphisms on a countable dimensional vector space (Example \ref{ex:countable}).

\section{Preliminaries} \label{s:pre}

\subsection{The Classification Problem}\label{ss:Cl-Pr}

Let $V$ be an arbitrary $k$-vector space, and let $\ed_k (V)$ be the $k$-vector space of endomorphisms of $V$.

We have an action of the group of automorphisms of $V$, $\aut_k (V)$, on $\ed_k (V)$ by conjugation:
$$\begin{aligned} \aut_k (V) \times \ed_k (V) &\longrightarrow \ed_k (V) \\ (\tau, f) &\longmapsto \tau f {\tau}^{-1} \, .\end{aligned}$$

Let us consider a subset $X \subset \ed_k (V)$ that is invariant under the action of $\aut_k (V)$ by conjugation.

The classification problem on $X$ refers to the possible answer to the question: which is the characterization of the quotient set $X \big / {\aut_k(V)}$?.

Henceforth, if $f\in \ed_k (V)$ and $H$ is a $k$-subspace of $V$ invariant by $f$, to simplify we shall again write $f\colon H \longrightarrow H$ and $f\colon V/H \longrightarrow V/H$ to refer to the induced linear operators.

\subsection{Classification of Endomorphisms on Finite-Dimensional Vector Spa\-ces}\label{ss:clas-finite}
\label{ss:ACK}

The solution of the classification problem for endomorphisms on vector spaces of finite dimension is well-known.

Let $E$ be a finite-dimensional vector space over a field $k$, and let $T\in \ed_k(E)$ be an endomorphism of $E$. We have that $T$ induces a structure of $k[x]$-module from the action $$\begin{aligned} k[x] \times E &\longrightarrow E \\ [p(x), e] &\longmapsto p(T)e \, .\end{aligned}$$\noindent We shall write $E_T$ to denote the vector space $E$ with this $k[x]$-module structure.

It is known that two endomorphisms $T, {\tilde T} \in \ed_k(E)$ are equivalent, i. e. there exists an automorphism $\tau \in \aut_k (V)$ such that $T = \tau {\tilde T} \tau^{-1}$, if and only if the $k[x]$-modules $E_T$ and $E_{{\tilde T}}$ are isomorphic.

The above action, which determines de $k[x]$-module structure of $E$ (via $T$), is equivalent to a morphism of rings $$\begin{aligned} \phi_T \colon k[x] &\longrightarrow \ed_k (E) \\ p(x) &\longmapsto \phi_T [p(x)]\, ,\end{aligned}$$\noindent where $\phi_T [p(x)] (e) = p(T) (e)$.

Since $\Ker \phi_T \ne \{0\}$, there exists a unique monic polynomial $a_T(x)$ such that $\Ker \phi_T = (a_T(x))$, $a_T(x)$ being the annihilator polynomial of $T$.

Let $a_T(x) = p_1(x)^{n_1}\cdot \dots \cdot p_r(x)^{n_r}$ be the annihilator polynomial of $T$, where $p_i(x)$ are irreducible polynomials on $k[x]$. The decomposition of $k$-vector spaces $$E = \Ker p_1(T)^{n_1}\oplus \dots \oplus \Ker p_r(T)^{n_r}\, ,$$\noindent where the subspaces $\Ker p_i(T)^{n_i} \subset E$ are invariant by $T$, is compatible with the respective $k[x]$-module structures.

Indeed, the classification of endomorphisms on finite-dimensional vector spaces is reduced to studying the $k[x]$-module structure of $\Ker p(T)^n$, $p(x)$ being an \linebreak irreducible polynomial on $k[x]$. This structure is determined by a decomposition of $k[x]$-modules:

$$\Ker p(T)^n \simeq {\big [ k[x] \big / p(x)^n \big ]}^{\nu_n (E,p(T))} \oplus \dots \oplus {\big [ k[x] \big / p(x) \big ]}^{\nu_1 (E,p(T))}\, ,$$\noindent where $\nu_n (E,p(T)) \ne 0$ and $$\nu_i (E,p(T)) = \dim_K \big ( \Ker p(T)^i \big / [\Ker p(T)^{i-1} + p(T) \Ker p(T)^{i+1}] \big )\, ,$$\noindent with $K = k[x] \big / p(x)$.

Again writing the annihilator polynomial of $T$ as $a_T(x) = p_1(x)^{n_1}\cdot \dots \cdot p_r(x)^{n_r}$, the invariant factors $\{\nu_i (E,p_j(T))\}_{1\leq j \leq r\, ; \, 1\leq i \leq n_j}$ determine the $k[x]$-module structure of $E_T$ and, therefore, they classify the endomorphism $T$.

Furthermore, if $K_j = k[x] \big / p_j(x)$, then $p_j(x)$ is a polynomial of degree $d_j = \dim_k K_j$.

We shall now describe a method for constructing Jordan bases of $E$ for $T$.

Let us first assume that $a_T(x) = p(x)^{n}$, with $p(x)$ an irreducible polynomial on $k[x]$, and let again $K = k[x]/p(x)$, with $d = \dim_k K = \text{ gr} p(x)$.

For each $1\leq i \leq n$, let $\{{\bar e}^{i}_h\}_{1\leq h \leq \nu_i (E,p(T))}$ be a basis of $$\Ker p(T)^i \big / [\Ker p(T)^{i-1} + p(T) \Ker p(T)^{i+1}]$$\noindent as a $K$-vector space.

If we write $$\pi_i \colon \Ker p(T)^i \longrightarrow \Ker p(T)^i \big / [\Ker p(T)^{i-1} + p(T) \Ker p(T)^{i+1}]$$\noindent to denote the natural projection, let $\{{e}^{i}_h\}_{1\leq h \leq \nu_i (E,p(T))}$ be a family of vectors of $\Ker p(T)^i$ such that $\pi_i ({e}^{i}_h) = {\bar e}^{i}_h$ for all $1\leq h \leq \nu_i (E,p(T))$.

\begin{rem} Since both $\Ker p(T)^i$ and $\Ker p(T)^{i-1} + p(T) \Ker p(T)^{i+1}$ are $k$-vector subspaces of $E$ invariant by $T$, then $\Ker p(T)^{i-1} + p(T) \Ker p(T)^{i+1}$ is also a $k[x]$-submodule of $(\Ker p(T)^i)_T$. We should emphasize that the definition of the quotient set $\Ker p(T)^i \big / [\Ker p(T)^{i-1} + p(T) \Ker p(T)^{i+1}]$ is independent of its structure as a $k$-vector space, $K$-vector space or $k[x]$-module (via $T$).
\end{rem}

\begin{lem} \label{lem:adsfst} If $\dim_k K = d$, then $\bigcup_{1\leq h \leq \nu_i (E,p(T))} \{{e}^{i}_h, T ({e}^{i}_h), T^{2} ({e}^{i}_h), \dots , T^{d-1} ({e}^{i}_h)\}$ is a linearly independent family of vectors of the $k$-vector space $\Ker p(T)^i$.
\end{lem}

\begin{proof} Let us consider a linear combination $$\sum_{\begin{aligned} 0 \leq &m \leq d-1 \\ 1 \leq &l \leq \nu_i (E,p(T)) \\ \end{aligned}} \lambda_{ml} T^m ({e}^{i}_l) = 0\, ,$$\noindent with $\lambda_{hl} \in k$.

Hence, since $$\sum_{\begin{aligned} 0 \leq &m \leq d-1 \\ 1 \leq &l \leq \nu_i (E,p(T)) \end{aligned}} \lambda_{ml} x^m {e}^{i}_l = 0$$ as a $k[x]$-module, if $\lambda_{ml} \ne 0$ in some case then $\{{\bar e}^{i}_h\}_{1\leq h \leq \nu_i (E,p(T))}$ will be a family of linearly dependent vectors of $\Ker p(T)^i \big / [\Ker p(T)^{i-1} + p(T) \Ker p(T)^{i+1}]$ as a $K$-vector space, which is impossible. Therefore, the statement is deduced.
\end{proof}

\begin{prop} \label{prop:asader} One has that  $$H^{p(T)}_i = \bigoplus_{1\leq h \leq \nu_i (E,p(T))} <{e}^{i}_h, T ({e}^{i}_h), T^{2} ({e}^{i}_h), \dots , T^{d-1} ({e}^{i}_h)>$$\noindent is a supplementary subspace of $\Ker p(T)^{i-1} + p(T) \Ker p(T)^{i+1}$ on the $k$-vector space $\Ker p(T)^i$.
\end{prop}

\begin{proof} It follows from Lemma \ref{lem:adsfst} that the dimension of $H^{p(T)}_i$ as a $k$-vector space is $d\cdot \nu_i (E,p(T))$, which coincides with $$\begin{aligned} &\dim_k \big ( \Ker p(T)^i \big / [\Ker p(T)^{i-1} + p(T) \Ker p(T)^{i+1}] \big ) = \\ &\dim_k [\Ker p(T)^i] - \dim_k [\Ker p(T)^{i-1} + p(T) \Ker p(T)^{i+1}]\, .\end{aligned}$$

Thus, to prove the claim it is sufficient to check that $$[\Ker p(T)^{i-1} + p(T) \Ker p(T)^{i+1}] \cap H^{p(T)}_i = \{0\}\, .$$

Let us now consider $e\in [\Ker p(T)^{i-1} + p(T) \Ker p(T)^{i+1}] \cap H^{p(T)}_i$. One has that $$e = \sum_{\begin{aligned} 0 \leq &m \leq d-1 \\ 1 \leq &l \leq \nu_i (E,p(T)) \end{aligned}} \gamma_{ml} T^m ({e}^{i}_l) \in [\Ker p(T)^{i-1} + p(T) \Ker p(T)^{i+1}]\, ,$$ and, bearing in mind that $\pi_i (e) = 0$, we conclude that $\gamma_{ml} = 0$ (for all $l,m$) and, thereby, $e = 0$ because, otherwise, $\{{\bar e}^{i}_h\}_{1\leq h \leq \nu_i (E,p(T))}$ will be a family of linearly dependent vectors of $\Ker p(T)^i \big / [\Ker p(T)^{i-1} + p(T) \Ker p(T)^{i+1}]$ as a $K$-vector space, which is impossible.
\end{proof}

Accordingly, bearing in mind that $$\Ker T^{i-1} \oplus p(T)^{n-i} H_n^{p(T)} \oplus \dots \oplus p(T) H_{i+1}^{p(T)} = \Ker p(T)^{i-1} + p(T)\Ker p(T)^{i+1}$$\noindent for all $1\leq i \leq n$, we have constructed a family of $k$ subspaces of $E = \Ker p(T)^n$, $\{H^{p(T)}_i\}_{1\leq i \leq n}$, such that
\begin{equation} \label{eq:decomp-anh} \begin{aligned} E &= \Ker p(T)^{n-1} \oplus H^{p(T)}_n \\ \Ker p(T)^{n-1} &= \Ker p(T)^{n-2} \oplus p(T)H^{p(T)}_n  \oplus H^{p(T)}_{n-1} \\ &\vdots \\ \Ker p(T)^{i} &= \Ker p(T)^{i-1}\oplus p(T)^{n-i}H^{p(T)}_n\oplus \dots \oplus p(T)H^{p(T)}_{i+1}  \oplus H^{p(T)}_{i} \\ &\vdots \\ \Ker p(T)^{2} &= \Ker p(T)\oplus p(T)^{n-2}H^{p(T)}_n\oplus \dots \oplus p(T)H^{p(T)}_{3}  \oplus H^{p(T)}_{2} \\ \Ker p(T) &=  p(T)^{n-1}H^{p(T)}_{n} \oplus p(T)^{n-2}H^{p(T)}_{n-1} \oplus \dots \oplus p(T)H^{p(T)}_{2} \oplus H^{p(T)}_{1}\, .  \end{aligned}\end{equation}

Thus, if we now write $$<e^{i}_h>_T = \underset {0\leq s \leq i - 1} \bigoplus  <p(T)^{s}[e^{i}_h], p(T)^{s} [T(e^{i}_h)], \dots , p(T)^{s}[T^{d -1}(e^{i}_h)]>\, ,$$\noindent then
$$E = \underset {\begin{aligned} 1 &\leq i \leq n \\ 1 \leq h &\leq \nu_i (E,p(T)) \end{aligned}} {\bigoplus} <e^{i}_h>_T\, .$$

In general, if $T\in \ed_k (E)$ with annihilator polynomial $$a_T(x) = p_1(x)^{n_1}\cdot \dots \cdot p_r(x)^{n_r}\, ,$$\noindent we have shown that there exist families of vectors $\{e^{ij}_h\}_{1\leq h \leq \nu_i (E,p_j(T))}$ with
 $$\begin{aligned} e^{ij}_h &\in \Ker p_j(T)^i \\ e^{ij}_h &\notin \Ker p_j(T)^{i-1} + p_j(T) \Ker p_j(T)^{i+1}\, ,\end{aligned}$$\noindent for all $1\leq j \leq r$ and $1\leq i \leq n_j$, such that if we set
 $$<e^{ij}_h>_T = \underset {0\leq s \leq i - 1} \bigoplus  <p_j(T)^{s}[e^{ij}_h], p_j(T)^{s} [T(e^{ij}_h)], \dots , p_j(T)^{s}[T^{d_j -1}(e^{ij}_h)]>\, ,$$\noindent then
$$E = \underset {\begin{aligned}1 &\leq j \leq r  \\ 1 &\leq i \leq n_j \\ 1 \leq h &\leq \nu_i (E,p_j(T)) \end{aligned}} {\bigoplus} <e^{ij}_h>_T\, .$$

Indeed, each family of vectors \begin{equation} \label{eq:bases-j-finite} \underset {\begin{aligned}1 &\leq j \leq r  \\ 1 &\leq i \leq n_j  \end{aligned}} {\bigcup} \{e^{ij}_h\}_{1\leq h \leq \nu_i (E,p_j(T))}\end{equation} generates a Jordan basis of $E$ for $T$.

\section{Classification of endomorphisms with an annihilating polynomial} \label{s:main-classif}

Let $V$ again be an arbitrary vector space over a ground field $k$, and let ${\mathcal B} =\{v_i\}_{i\in I}$ be a basis of $V$. It is known that $\dim (V) = \# {\mathcal B}$ is independent of the basis chosen, $\# {\mathcal B}$ being the cardinal of the set ${\mathcal B}$.

Let $X^{anh}_V$ be the subset of $\ed_k (E)$ consisting of all endomorphisms of $V$ with an annihilating polynomial, that is:
$$X^{anh}_V = \{f \in \ed_k (V) \text{ such that } p(f) = 0 \text{ for a certain } p(x)\in k[x]\, \}\, .$$

We should note that $X^{anh}_V$ is not a vector subspace of $\ed_k (E)$ (in general the sum of two endomorphisms with an annihilating polynomial does not admit an annihilating polynomial).

It is clear that the group of automorphisms of $V$, $\aut_k(V)$, acts on $X^{anh}_V$ by conjugation, because if $p(f) = 0$, then $p(\tau f \tau^{-1}) = 0$ for all $\tau \in \aut_k(V)$.

Let $\phi_{\varphi} \colon k[x] \longrightarrow \ed_k(V)$ be the morphism of rings that induces the $k[x]$-module structure in $V$. If $\varphi$ admits an annihilating polynomial, then $\Ker \phi_{\varphi} \ne \{0\}$ and there exists a unique monic polynomial $a_{\varphi} (x)$ such that $\Ker \phi_{\varphi} = (a_{\varphi} (x))$. The polynomial $a_{\varphi} (x)$ is the annihilator of $\varphi$.

The aim of this section is to offer the main result of this work: i.e. to provide the characterization of the quotient set $X^{anh}_V\big /{\aut_k(V)}$.

\begin{rem} An endomorphism $\varphi \in \ed_k (V)$ is ``finite potent''  if $\varphi^n V$ is finite
dimensional for some $n$. It is clear that a finite potent endomorphism admits an annihilating polynomial, and if $X^{fp}_V$ is the set consisting of all finite potent endomorphisms of $V$, then $X^{fp}_V$ is a subset of $X^{anh}_V$ that is invariant under the action of $\aut_k (V)$ by conjugation. We should note that the classification problem for finite potent endomorphisms has recently been solved by the author in \cite{Pa}.
\end{rem}

To start, we shall construct a Jordan Basis of $V$ for an endomorphism with an annihilating polynomial by generalizing the method described in Subsection \ref{ss:clas-finite}.

Let us consider $f \in X^{anh}_V$, and we first assume that $a_f(x) = p(x)^n$, $p(x)$ being an irreducible polynomial on $k[x]$.

As in Subsection \ref{ss:clas-finite}, we consider
$$\nu_i (V,p(f)) = \dim_K \big ( \Ker p(f)^i \big / [\Ker p(f)^{i-1} + p(f) \Ker p(f)^{i+1}] \big )\, ,$$\noindent with $K = k[x] \big / p(x)$.

Henceforth, for indexing bases $S_{{\nu}_i (V,p(f))}$ will be a set such that $\# S_{{\nu}_i (V,p(f))} = {\nu}_i (V,f)$, with $S_{{\nu}_i (V,p(f))} \cap S_{{\nu}_j (V,p(f))} = \emptyset$ for $i\ne j$. In particular, if ${\nu}_i (V,p(f))$ is a natural number $N$, then $S_{{\nu}_i (V,p(f))} = \{i_1,i_2, \dots,i_N\}$.

Similar to above, for each $1\leq i \leq n$, let $\{{\bar v}^{i}_h\}_{h \in S_{{\nu}_i (V,p(f))}}$ be a basis of $$\Ker p(f)^i \big / [\Ker p(f)^{i-1} + p(f) \Ker p(f)^{i+1}]$$\noindent as a $K$-vector space.

If we again write $$\pi_i \colon \Ker p(f)^i \longrightarrow \Ker p(f)^i \big / [\Ker p(f)^{i-1} + p(f) \Ker p(f)^{i+1}]$$\noindent to denote the natural projection, let $\{{v}^{i}_h\}_{h \in S_{{\nu}_i (V,p(f))}}$ be a family of vectors of $\Ker p(f)^i$ such that $\pi_i ({v}^{i}_h) = {\bar v}^{i}_h$ for all $h \in S_{{\nu}_i (V,p(f))}$.

\begin{lem} \label{lem:adsfst-2} If $\dim_k K = d$, then $\bigcup_{h \in S_{{\nu}_i (V,p(f))}} \{{v}^{i}_h, f ({v}^{i}_h), f^{2} ({v}^{i}_h), \dots , f^{d-1} ({v}^{i}_h)\}$ is a linearly independent family of vectors of the $k$-vector space $\Ker p(f)^i$.
\end{lem}

\begin{proof} With the same arguments as in Lemma \ref{lem:adsfst}, if $L\subseteq S_{{\nu}_i (V,p(f))}$, let us consider a linear combination $$\sum_{l\in L} \big [ \sum_{0 \leq m \leq d-1} \lambda_{ml} f^m ({v}^{i}_l) \big ] = 0\, ,$$\noindent with $\lambda_{hl} \in k$, $\underset {l\in L} \sum$ being a well-defined expression in $\Ker p(f)^i$.

Hence, since $$\sum_{l\in L} \big [ \sum_{0 \leq m \leq d-1} \lambda_{ml} x^m {v}^{i}_l \big ] = 0$$ as a $k[x]$-module, if $\lambda_{ml} \ne 0$ in some case, then $\{{\bar v}^{i}_h\}_{h \in S_{{\nu}_i (V,p(f))}}$ will be a family of linearly dependent vectors of $\Ker p(f)^i \big / [\Ker p(f)^{i-1} + p(f) \Ker p(f)^{i+1}]$ as a $K$-vector space, which is impossible. Therefore, the statement is deduced.
\end{proof}

\begin{prop} \label{prop:asader-2} If $H^{p(f)}_i$ is the $k$-subspace of $\Ker p(f)^i$ generated by the family of vectors $$\bigcup_{h \in S_{{\nu}_i (V,p(f))}} \{{v}^{i}_h, f ({v}^{i}_h), f^{2} ({v}^{i}_h), \dots , f^{d-1} ({v}^{i}_h)\}\, ,$$\noindent then $H^{p(f)}_i$ is a supplementary subspace of $\Ker p(f)^{i-1} + p(f) \Ker p(f)^{i+1}$ on the $k$-vector space $\Ker p(f)^i$.
\end{prop}

\begin{proof} Given a vector $v\in \Ker p(f)^i$, considering the structure of $$\Ker p(f)^i \big / [\Ker p(f)^{i-1} + p(f) \Ker p(f)^{i+1}]$$\noindent as a $K$-vector space, one has that $$\pi_i (v) = \sum_{l\in L} \mu_l q_l(x) {\bar v}^i_l\, ,$$\noindent with $q_l(x) = a_0^l + a_1^l x + \dots + a_{d-1}^l x^{d-1} \in k[x]/p(x) = K$, $L\subseteq S_{{\nu}_i (V,p(f))}$, $\underset {l\in L} \sum$ being a well-defined expression in $\Ker p(f)^i \big / [\Ker p(f)^{i-1} + p(f) \Ker p(f)^{i+1}]$.

Since from the definition of the sum on the quotient set we deduce that $\underset {l\in L} \sum$ is also a well-defined expression in $\Ker p(f)^i$, bearing in mind that $$\pi_i \big ( v - \sum_{l\in L} \big [ \sum_{0 \leq m \leq d-1} \mu_l a_m^l f^m ({v}^{i}_l) \big ] \big ) = 0\, ,$$ it clear that $\Ker p(f)^i = [\Ker p(f)^{i-1} + p(f) \Ker p(f)^{i+1}] + H^{p(f)}_i$.

Moreover, analogously to Proposition \ref{prop:asader}, it is easy to check that $$[\Ker p(f)^{i-1} + p(f) \Ker p(f)^{i+1}] \cap H^{p(f)}_i = \{0\}\, ,$$\noindent and the claim is proved.
\end{proof}

    Thus, recurrently, as in expressions (\ref{eq:decomp-anh}), we can consider $k$-vector subspaces of $V$, $\{H_{r}^{p(f)}\}_{1\leq r \leq n}$, such that:
\begin{equation} \label{eq:adad} \Ker f^i = \Ker f^{i-1} \oplus p(f)^{n-i} H_n^{p(f)} \oplus \dots \oplus p(f)H_{i+1}^{p(f)} \oplus H_i^{p(f)}\, ,\end{equation}\noindent for every $1\leq i < n$.

\begin{rem} Recall from \cite{Pa} that there is no relationship of order between the invariants $\{\nu_i(V,p(f))\}$.

    For example, if $V$ is a $k$-vector space of countable dimension with a basis $\{e_1,e_2,\dots,e_n,\dots\}$, and we consider $f,g \in \in X^{anh}_V$ defined by:
$$f(e_i) = \left \{ {\begin{aligned} e_3  &\text{ if } i = 1,2 \\ 0  &\text{ if } i \geq 3 \end{aligned}} \right . \, ,$$

$$g(e_j) = \left \{ {\begin{aligned} 0 &\text{ if } j = 1,2 \\ e_{j+1} &\text{ if } j \geq 3 \text{ odd } \\ 0 &\text{ if } j \geq 4 \text{ even } \end{aligned}} \right . \, ,$$\noindent then
$a_f(x) = a_g (x) = x^2$, and one has that:
\begin{itemize}
\item  $H_2^f = <e_1>$, $f(H_2^f) = <e_3>$, and $H_1^f = <e_1 - e_2, e_4, e_5, ...>$;
\item $\nu_1 (V,f) ) = \aleph_0$, and $\nu_2 (V,f) ) = 1$;
\item $H_2^g = <e_{2j + 1}>_{j\geq 1}$, $g(H_2^g) = <e_{2j}>_{j\geq 2}$, and $H_1^g = <e_1,e_2>$;
\item $\nu_1 (V,2) ) = 2$, and $\nu_2 (V,f) ) = \aleph_0$;
 \end{itemize}
$\aleph_0$ being the cardinal of the set of all natural numbers.

\end{rem}

    If we now write $$<v^{i}_h>_f = \underset {0\leq s \leq i - 1} \bigcup  \{p(f)^{s}[v^{i}_h], p(f)^{s} [f(v^{i}_h)], \dots , p(f)^{s}[f^{d -1}(v^{i}_h)]\}\, ,$$\noindent then
$$\underset {\begin{aligned} 1 &\leq i \leq n \\  h &\in S_{{\nu}_i (V,p(f))} \end{aligned}} {\bigcup} <v^{i}_h>_f$$\noindent is a Jordan basis of $V$ for $f$.

In general, if $f\in X^{anh}_V$ with annihilator polynomial $$a_f(x) = p_1(x)^{n_1}\cdot \dots \cdot p_r(x)^{n_r}\, ,$$\noindent we have shown that there exist families of vectors $\{v^{ij}_h\}_{ h \in S_{{\nu}_i (V,p_j(f))}}$ with
 $$\begin{aligned} v^{ij}_h &\in \Ker p_j(f)^i \\ v^{ij}_h &\notin \Ker p_j(f)^{i-1} + p_j(f) \Ker p_j(f)^{i+1}\, ,\end{aligned}$$\noindent for all $1\leq j \leq r$ and $1\leq i \leq n_j$, such that if we set
 $$<v^{ij}_h>_f = {\underset {0\leq s \leq i - 1} \bigcup}  \{p_j(f)^{s}[v^{ij}_h], p_j(f)^{s} [f(v^{ij}_h)], \dots , p_j(f)^{s}[f^{d_j -1}(v^{ij}_h)]\}\, ,$$\noindent then
$$\underset {\begin{aligned}1 &\leq j \leq r  \\ 1 &\leq i \leq n_j \\  h &\in S_{{\nu}_i (V,p_j(f))} \end{aligned}} {\bigcup} <v^{ij}_h>_f$$\noindent is a Jordan basis of $V$ for $f$.

Indeed, each family of vectors \begin{equation} \label{eq:bases-j-infinite} \underset {\begin{aligned} 1 &\leq j \leq r  \\ 1 &\leq i \leq n_j  \end{aligned}} {\bigcup} \{v^{ij}_h\}_{h \in S_{{\nu}_i (V,p_j(f))}}\end{equation} generates a Jordan basis of $V$ for $f$.

 Note that $$\# \big [ \underset {\begin{aligned} 1 &\leq j \leq r  \\ 1 &\leq i \leq n_j  \end{aligned}} {\bigcup} (S_{{\nu}_i (V,p_j(f))} \sqcup \overset {[d-1]i)} {\dots \dots \dots} \sqcup S_{{\nu}_i (V,p_j(f))}) \big ] = \dim (V)\, .$$

\begin{rem}
Recall from \cite{LB} that for every  $\phi \in \ed_k (V)$
possessing an annihilating polynomial of an arbitrary
infinite-dimensional vector space $V$ there exists a Jordan basis
of $V$ associated with $\phi$. We should note that the above construction of Jordan bases for endomorphisms with an annihilating polynomial is compatible with the results of \cite{LB}. However, from the proof of the existence of Jordan bases given in \cite{LB} a Classification Theorem for these endomorphisms can not be obtained, because from the statements of this paper it is not possible to deduce that the dimensions of the vector subspaces that determine a Jordan basis are independent of the choices made.
\end{rem}

We shall now use the existence of Jordan bases for endomorphisms with an annihilating polynomial to characterize the quotient set $X^{anh}_V\big /{\aut_k(V)}$.

Let $\tau$ be an automorphism of $V$.

\begin{lem} \label{lem:palace} If $v\in V$, then $$p(f)^s (v) = 0 \Longleftrightarrow p({\bar f})^s (\tau (v)) = 0$$\noindent with ${\bar f} = \tau f \tau^{-1}$.
\end{lem}

\begin{proof} One has that: $$p({\bar f})^s (\tau (v)) \Longleftrightarrow \tau [p(f)^s (v)] = 0 \Longleftrightarrow p(f)^s (v) = 0\, .$$
\end{proof}

\begin{cor} \label{cor:palace} If ${\bar f} = \tau f \tau^{-1}$, then
\begin{itemize}
\item $\tau [\Ker p(f)^r] = \Ker p({\bar f})^r$ for all $r\geq 1$ and $p(x) \in k[x]$.
\item $\tau [\Ker p(f)^{i-1} + p(f)\Ker p(f)^{i+1}] = \Ker p({\bar f})^{i-1} + p({\bar f})\Ker p({\bar f})^{i+1}$ for all $i\geq 1$ and $p(x)\in k[x]$.
\end{itemize}
\end{cor}

\begin{prop} \label{prop:avila} Let $f\in X^{anh}_V$ with $a_f(x) = p_1(x)^{n_1}\cdot \dots \cdot p_r(x)^{n_r}$, $p_i (x)$ being a irreducible polynomial in $k[x]$. If ${\bar f} = \tau f \tau^{-1}$. Then:
\begin{itemize}
\item $a_{\bar f}(x) = a_f(x)$.

\item ${\nu}_i (V,p_j(f)) = {\nu}_i (V,{p_j(\bar f)})$ for all $1 \leq j \leq r$ and $1\leq i \leq n_j$.
\end{itemize}
\end{prop}

\begin{proof}
Given $q(x) \in k[x]$, it is clear that $$q({\bar f}) = 0 \Longleftrightarrow q(f) = 0\, ,$$
and hence $a_{\bar f}(x) = a_f(x)$. In particular, ${\bar f}\in X^{anh}_V$.

    On the other hand, for all $1 \leq j \leq r$ and $1\leq i \leq n_j$, bearing in mind Corollary \ref{cor:palace}, one has that

\xymatrix{
0 \ar[r] & \Ker p_j(f)^{i-1} + p_j(f)(\Ker p_j(f)^{i+1}) \ar[r] \ar[d]^{\sim}_{\tau} &  \Ker p_j(f)^i \ar[d]_{\sim}^{\tau}  \\
0 \ar[r] & \Ker p_j({\bar f})^{i-1} + p_j({\bar f})(\Ker p_j({\bar f})^{i+1}) \ar[r] &  \Ker p_j({\bar f})^i
}

and hence $\tau$ induces an isomorphism of $k$-vector spaces

\begin{equation} \label{eq: ismo-basic} \xymatrix{\Ker p_j(f)^i  \big / [\Ker p_j(f)^{i-1} + p_j(f)\Ker p_j(f)^{i+1}] \ar[d]^{\sim}_{\tau} \\ \Ker p_j({\bar f})^i  \big / [\Ker p_j({\bar f})^{i-1} + p_j({\bar f})(\Ker p_j({\bar f})^{i+1})]\, .} \end{equation}

Furthermore, for each family of vectors $\{w_s\}_{s\in I} \subset \Ker p_j(f)^i$, it is clear that $$\tau [\sum_{r_s \in {\mathbb N}; s\in I} f^{r_s} (w_s)] = \sum_{r_s \in {\mathbb N}; s\in I} {\bar f}^{r_s} (\tau [w_s])\, .$$

Thus, if we consider the structures of $k[x]$-module induced in $\Ker p_j(f)^i$ by $f$ and in $\Ker p_j({\bar f})^i$ by ${\bar f}$ respectively, we have that
$$\tau [\sum_{r_s \in {\mathbb N}; s\in I} x^{r_s} (w_s)] = \sum_{r_s \in {\mathbb N}; s\in I} x^{r_s} (\tau [w_s])\, ,$$\noindent
and, therefore, expression (\ref{eq: ismo-basic}) is an isomorphism of $K$-vector spaces, from which the statement can be deduced.
\end{proof}

\begin{thm}[Classification Theorem] \label{th:class-anh} Let $f, g \in X^{anh}_V$ be two endomorphisms with an annihilating polynomial. Thus, $f \sim g$ (mod. $\aut_k(V)$) if and only if:
\begin{itemize}
\item $a_{g}(x) = a_f(x) = p_1(x)^{n_1}\cdot \dots \cdot p_r(x)^{n_r}$, $p_i (x)$ being a irreducible polynomial in $k[x]$.

\item ${\nu}_i (V,p_j(f)) = {\nu}_i (V,{p_j(g)})$ for all $1 \leq j \leq r$ and $1\leq i \leq n_j$.
\end{itemize}
\end{thm}

\begin{proof}

$\Longrightarrow )$ If $[f] = [g] \in X^{anh}_V\big /{\aut_k(V)}$, there exists $\tau \in \aut_k (V)$ such that $g = \tau f \tau^{-1}$, and it follows from Proposition \ref{prop:avila} that:

\begin{itemize}
\item $a_{g}(x) = a_f(x) = p_1(x)^{n_1}\cdot \dots \cdot p_r(x)^{n_r}$, $p_i (x)$ being a irreducible polynomial in $k[x]$.

\item ${\nu}_i (V,p_j(f)) = {\nu}_i (V,{p_j(g)})$ for all $1 \leq j \leq r$ and $1\leq i \leq n_j$.
\end{itemize}

$\Longleftarrow )$ If $a_{g}(x) = a_f(x) = p_1(x)^{n_1}\cdot \dots \cdot p_r(x)^{n_r}$ and ${\nu}_i (V,p_j(f)) = {\nu}_i (V,{p_j(g)})$ for all $1 \leq j \leq r$ and $1\leq i \leq n_j$, let us consider two families of vectors $$\underset {\begin{aligned} 1 &\leq j \leq r  \\ 1 &\leq i \leq n_j  \end{aligned}} {\bigcup} \{v^{ij}_h\}_{h \in S_{{\nu}_i (V,p_j(f))}} \text{ and } \underset {\begin{aligned} 1 &\leq j \leq r  \\ 1 &\leq i \leq n_j  \end{aligned}} {\bigcup} \{w^{ij}_h\}_{h \in S_{{\nu}_i (V,p_j(g))}}$$\noindent determining Jordan bases of $V$ for $f$ and $g$ respectively.

Let $\tau \in \aut_k(V)$ be the automorphism defined by $$\tau [p_j(f)^m (f^a (v^{ij}_h))] = p_j(g)^m (g^a (w^{ij}_h))\, ,$$\noindent for all  $1\leq j \leq r$, $1\leq i \leq n_j$, $h \in S_{{\nu}_i (V,p_j(g))}$, $0\leq a \leq d_j-1$, and $0\leq m < i$.

An easy computation shows that
$$\tau [p_j(f)^m (f^b (v^{ij}_h))] = p_j(g)^m (g^b (w^{ij}_h))\, ,$$\noindent for all $b\geq 0$.

By construction, one has that $$\begin{aligned} (\tau f \tau^{-1})[p_j(g)^m (g^a (w^{ij}_h))] &= (\tau f) [p_j(f)^m (f^a (v^{ij}_h))] = \tau [p_j(f)^m (f^{a+1} (v^{ij}_h))] = \\ &=  p_j(g)^m (g^{a+1} (v^{ij}_h)) = g[p_j(g)^m (g^{a} (v^{ij}_h))] \, ,\end{aligned}$$\noindent for all  $1\leq j \leq r$, $1\leq i \leq n_j$, $h \in S_{{\nu}_i (V,p_j(g))}$, $0\leq a \leq d_j-1$, and $0\leq m < i$.

Accordingly, $\tau f \tau^{-1} = g$ and $[f] = [g] \in X^{anh}_V\big /{\aut_k(V)}$.
\end{proof}

\begin{exam} \label{ex:countable} Let $V$ be a countable dimensional $k$-vector space. Let $$Y = \{0\} \cup {\mathbb N} \cup \{\aleph_0\}\, ,$$\noindent $\aleph_0$ being the cardinal of the set of all natural numbers, and let $H$ be the set consisting of all monic irreducible polynomials $p(x) \in k[x]$.

If $H_r \subset H\times \overset {r)} {\dots \dots} \times H$ is the set such that $$(p_1(x), \dots , p_r(x)) \in H_r \Longleftrightarrow p_j(x) \ne p_h(x) \text{ for all } j\ne h\, ,$$\noindent and $$Y_{{\mathbb N}} = \{(n,\nu_1, \dots , \nu_n) \text{ with } n\in {\mathbb N} \text{ and } \nu_i \in Y\}\, ,$$\noindent it follows from Theorem \ref{th:class-anh} that
$$X^{anh}_V\big /{\aut_k(V)} = \underset {r\in {\mathbb N}} \bigcup \big ( X_r \times [ {\prod_r}' Y_{{\mathbb N}}] \big )\, ,$$\noindent where $${\prod_r}' Y_{{\mathbb N}} = \left \{\begin{aligned} &(n_1, \nu_{11}, \dots , \nu_{1n_1}, \dots, n_i, \nu_{i1}, \dots , \nu_{in_i}, \dots ,n_r, \nu_{r1}, \dots , \nu_{rn_r}) \\ &\text{ with } \nu_{ji}\in Y\, , \, \nu_{in_i} \ne 0 \, , \text{ and } \nu_{ji} = \aleph_0  \text{ for at least one } \nu_{ji}\end{aligned} \right \}\, .$$
\end{exam}

\begin{rem} Recently, in \cite{HP}, D. Hern\'andez Serrano and the author have offered an algebraic definition of infinite determinants $\Det^k_V(1 +\varphi)$ on an arbitrary $k$-vector space $V$, $\varphi$ being a finite potent endomorphism. It is known that $\varphi$ is finite potent if and only if $a_{\varphi} (x) = x^m\cdot p(x)$, with $(x,p(x)) = 1$ and $\dim_k \Ker p(\varphi) < \infty$. Accordingly, $a_{1 + \varphi} (x) = (x-1)^m p(x-1)$, and since $$1 + \varphi \text{ is invertible } \Longleftrightarrow \Det^k_V(1 +\varphi) \ne 0\, ,$$\noindent it follows from the above classification and the statements of \cite{HP} that for each finite potent endomorphism $\varphi$ $1 + \varphi$ is invertible if and only if ${\nu}_i (V,1+\varphi) = 0$ for all $i>0$.
\end{rem}

\begin{rem}[Final Consideration] If $E$ is a finite-dimensional $k$ vector space, and $f\in \ed_k (E)$ with $a_f(x) = p_1(x)^{n_1}\cdot \dots \cdot p_r(x)^{n_r}$, $p_i (x)$ being a irreducible polynomial in $k[x]$, and invariants $\{{\nu}_i (V,p_j(f))\}_{1\leq j \leq r; 1\leq i \leq n_r}$, we should note that the structure of $E$ as a $k[x]$-module induced by $f$ is: $$E_f \simeq \underset {\begin{aligned} 1 &\leq j \leq r  \\ 1 &\leq i \leq n_j  \end{aligned}} \bigoplus {\big ( k[x]/p_j(x)^i \big )}^{{\nu}_i (V,p_j(f))} \ \, .$$
Thus, Theorem \ref{th:class-anh} offers a classification of endomorphisms with an annihilator polynomial on arbitrary vector spaces that generalizes the well-known classification of endomorphisms on finite-dimensional vector spaces.
\end{rem}

\end{document}